\documentclass{amsart}
\usepackage[margin=3.5cm]{geometry}
\usepackage{amssymb,amsthm}
\newtheorem{theorem}{Theorem}[section]
\newtheorem{proposition}[theorem]{Proposition}
\newtheorem{lemma}[theorem]{Lemma}
\theoremstyle{remark}
\newtheorem*{remark}{Remark}
\newcommand{\tO}{\mathtt 0}       
\newcommand{\tL}{\mathtt 1}       
\def\modd#1 #2{#1\ \mbox{\rm (mod}\ #2\mbox{\rm )}}
\allowdisplaybreaks
\begin{document}
\title[Modified Pascal triangle]%
{Continuants, run lengths, and Barry's modified Pascal triangle}
\author{Jeffrey Shallit}
\address{School of Computer Science, University of Waterloo, Waterloo, ON  N2L 3G1, Canada}
\email{shallit@cs.uwaterloo.ca}
\thanks{The first author is supported by a grant from NSERC.}
\author{Lukas Spiegelhofer}
\address{Institut f\"ur diskrete Mathematik und Geometrie,
Technische Universit\"at Wien,
Wiedner Hauptstrasse 8--10, 1040 Wien, Austria}
\email{lukas.spiegelhofer@tuwien.ac.at}
\thanks{The second author acknowledges support by 
project F5502-N26 (FWF), which is a part of the Special Research Program
``Quasi Monte Carlo Methods: Theory and Applications''.
Moreover, the second author thanks the Erwin Schr\"odinger Institute for Mathematics and Physics for providing the opportunity to carry out research work during his visit for the programme ``Tractability of High Dimensional Problems and Discrepancy''.
}
\begin{abstract}
We show that the $n$'th diagonal sum of Barry's modified Pascal triangle
can be described as the continuant of the run lengths of the binary
representation of $n$.  We also obtain an explicit description for
the row sums.
\end{abstract}
\maketitle
\section{Introduction}
In 2006 in the {\it On-Line Encyclopedia of Integer Sequences} (OEIS) 
\cite{OEIS},
sequence \texttt{A119326}, Paul Barry
introduced a modified Pascal triangle, defined for integers
$0 \leq k \leq n$, as follows:
\[T(n,k)=\sum_{\substack{0\leq j\leq n-k\\2\mid j}}\binom{k}{j}\binom{n-k}{j}.\]
The first few rows of this triangle are as follows:
\begin{gather*}
1 \\
1 \qquad 1 \\
1 \qquad 1 \qquad 1 \\
1 \qquad 1 \qquad 1 \qquad 1  \\
1 \qquad 1 \qquad 2 \qquad 1 \qquad 1 \\
1 \qquad 1 \qquad 4 \qquad 4 \qquad 1 \qquad 1 \\
1 \qquad 1 \qquad 7 \qquad 10 \qquad 7 \qquad 1 \qquad 1 \\
1 \qquad 1 \qquad 11 \qquad 19 \qquad 19 \qquad 11 \qquad 1 \qquad 1 \\
1 \qquad 1 \qquad 16 \qquad 31 \qquad 38 \qquad 31 \qquad 16 \qquad 1 \qquad 1
\end{gather*}
\newpage
Similarly, one can consider $T(n,k) \bmod 2$, whose terms are given by
sequence \texttt{A114213}:
\begin{gather*}
1 \\
1 \qquad 1 \\
1 \qquad 1 \qquad 1 \\
1 \qquad 1 \qquad 1 \qquad 1  \\
1 \qquad 1 \qquad 0 \qquad 1 \qquad 1 \\
1 \qquad 1 \qquad 0 \qquad 0 \qquad 1 \qquad 1 \\
1 \qquad 1 \qquad 1 \qquad 0 \qquad 1 \qquad 1 \qquad 1 \\
1 \qquad 1 \qquad 1 \qquad 1 \qquad 1 \qquad 1 \qquad 1 \qquad 1 \\
1 \qquad 1 \qquad 0 \qquad 1 \qquad 0 \qquad 1 \qquad 0 \qquad 1 \qquad 1
\end{gather*}

Sequences \texttt{A114212} and \texttt{A114214},
respectively, are the row sums and
diagonal sums of this latter triangle.  We denote them by
$r(n)$ and $d(n)$, respectively:
\begin{align*}
r(n) &= \sum_{k=0}^n \left( T(n,k) \bmod 2 \right) \\
d(n)&=\sum_{k=0}^{\lfloor n/2 \rfloor} \left( T(n-k,k) \bmod 2 \right).
\end{align*}

In May 2016, the first author observed, empirically, a connection between
$d(n)$ and the binary representation of $n$.  In this note we prove this
connection, and also prove a formula for $r(n)$.  The connection
involves Stern's ``diatomic sequence'' $s(n)$, defined by
$s(0) = 0$, $s(1) = 1$, $s(2n) = s(n)$, and
$s(2n+1) = s(n) + s(n+1)$; see \cite{Stern:1858}.
\section{The diagonal sums}

Let the binary representation of $n$ be denoted by
$\sum_{i=0}^j \varepsilon_i (n) 2^i$.
We show that the diagonal sum $d(n)$ can be expressed in terms of this
representation.
Given a string $s$ of $0$'s and $1$'s,
we consider its {\it run lengths}:  
the lengths of maximal blocks of consecutive identical elements.  For example,
if $s = 111000011111$, then the run lengths of $s$ are $(3,4,5)$. 

If $m$ is a sequence of positive integers, we may associate an
integer with it via the continued fraction expansion: if $m=(m_0,\ldots,m_k)$, we say that the \emph{continuant} of $m$ is the numerator of the 
continued fraction $[m_0;m_1,\ldots,m_k]$ (see \cite[Ch.~34, \S 4]{Chrystal}).
\begin{theorem}\label{mainthm}
Let $n\geq 0$ be an integer and let $m$ be the sequence of
run lengths of the binary representation of $n$.
Then $d(n)$ equals the continuant of $m$.
\end{theorem}

We will use Lucas' famous congruence for binomial
coefficients \cite[p.~230]{Lucas:1878}:
if $p$ is a prime number and $n=(n_\nu\cdots n_0)_p$ and $k=(k_\nu\cdots k_0)_p$, then
\[\binom nk\equiv \binom{n_\nu}{k_\nu}\cdots \binom{n_0}{k_0}\pmod p . \]
This implies that $\binom nk$ is not divisible by $p$ if and only if $k_i\leq n_i$ for all $i$.
Moreover, it follows that the number of odd binomial coefficients $\binom nk$ equals $2^{s_2(n)}$, where $s_2$ is the binary sum-of-digits function
\cite{Glaisher:1899}.

We prove the following statement, which reduces the problem to divisibility by $2$ of binomial coefficients.
We will derive Theorem~\ref{mainthm} from it in a moment.
\begin{proposition}\label{prp_reduction}
Let $n$ and $k$ be nonnegative integers such that $k\leq n$.
If $2\mid n+k$, then $T(n,k)\equiv \modd{\binom nk} {2}$.
Otherwise, $T(n,k)\equiv \modd{\binom{n-1}k} {2}$.
\end{proposition}
\begin{proof}
We prove the first statement. By replacing $n$ with $n+k$ we get
the equivalent assertion that if $2\mid n$ or $2\mid k$, then
\begin{equation}\label{eqn_sufficient}
\sum_{\substack{0 \leq j \leq n\\2\mid j}}
\binom nj\binom kj\equiv \binom {n+k}k\pmod 2 .
\end{equation}
By Lucas' congruence the left-hand side is congruent to
\[\sum_{j=0}^{n}\binom nj\binom kj\equiv\sum_{j=0}^{n}\binom {n\wedge k}j
\equiv 2^{s_2(n\wedge k)} \pmod 2,
\]
where $n\wedge k$ is the integer whose binary digits satisfy $\varepsilon_i(n\wedge k)=\min(\varepsilon_i(n),\varepsilon_i(k))$.
This expression is odd if and only if $s_2(n\wedge k)=0$, which is the case if and only if the binary representations of $n$ and $k$ are disjoint.
To handle the right-hand side of Eq.~\eqref{eqn_sufficient}, we note that $\binom{n+k}{k}$ is odd if $n\wedge k=1$.
On the other hand,
if the binary representations of $n$ and $k$ are not disjoint,
then the condition $\varepsilon_i(k)\leq \varepsilon_i(n+k)$ is 
violated for
$i=\min\{j:\varepsilon_j(n)=1,
\varepsilon_j(k)=1\}$; therefore $\binom{n+k}{k}$ is even.
This proves the first assertion.

For the second assertion, we use Lucas' congruence again: for $2\mid j$ and $2\mid m$ we have $\binom mj\equiv \binom{m+1}j\pmod 2$. Since $2\nmid n-k$, we obtain $\binom{n-k}j\equiv\binom{n-1-k}j\pmod 2$.
Moreover, by $2\nmid n-k$ the ranges of summation in $T(n,k)$ and $T(n-1,k)$ are the same.
\end{proof}
From this proposition we obtain in particular the identity
\begin{equation}\label{eqn_double}
d(2n)=d(2n+1).
\end{equation}

Carlitz \cite{Carlitz:1962} proved
that Stern's diatomic sequence $s(n)$ satisfies 
$s(n+1)=\sum_{k=0}^{\lfloor n/2\rfloor} \left( \binom{n-k}k \bmod 2 \right)$.
By Proposition~\ref{prp_reduction} and Eq.~\eqref{eqn_double}
we therefore have
\begin{equation}\label{eqn_stern_relation}
d(2n)=d(2n+1)=s(2n+1).
\end{equation}

It is well-known~\cite{L1929, L1969} that
if $m=(m_0,\ldots,m_k)$ is the sequence of run-lengths of the binary 
representation of $n$ and $n$ is odd,
then $s(n)$ is the continuant of $m$.
Therefore $d(n)$ is the continuant of $m$.
In order to complete the proof of the conjecture,
we have to show that the same is true for even $n$.
By Eq.~\eqref{eqn_stern_relation} it is sufficient to prove the following lemma.

\begin{lemma}
If $n$ is even, then the continuant of the sequence of run-lengths of the binary representation of $n$ is equal to the continuant corresponding to $n+1$.
\end{lemma}
\begin{proof}
Let $n=\tL^{m_0}\tO^{m_1}\cdots \tL^{m_{k-1}}\tO^{m_k}$.
We distinguish between two cases. If $m_k=1$, then
$n+1=(\tL^{m_0}\tO^{m_1}\cdots \tO^{m_{k-2}}\tL^{m_{k-1}+1})$
and the statement follows from the identity $[m_0;m_1,\ldots,m_{k-1},1]=[m_0;m_1,\ldots,m_{k-1}+1]$. If $m_k\geq 2$, then
$n+1=(\tL^{m_0}\tO^{m_1}\cdots \tO^{m_{k-2}}\tL^{m_{k-1}}\tO^{m_k-1}1)$ and the statement follows from 
$[m_0;m_1,\ldots,m_k]=[m_0;m_1,\ldots,m_{k-1},m_k-1,1]$.
\end{proof}

\begin{remark}
The sequence $(d(n))_{n\geq 0}$ is a $2$-regular sequence 
\cite{Allouche&Shallit:1992}, as it satisfies the equalities
\begin{align*}
d(2n+1) &= d(2n) \\
d(4n+2) &= 3d(2n) - d(4n) \\
d(8n) &= -d(2n) + 2 d(4n) \\
d(8n+4) &= 4d(2n) - d(4n) .
\end{align*}
\end{remark}
\section{The row sums}

We will prove

\begin{theorem}
$$r(n)= \begin{cases}
	2^{s_2 (n)}, & \text{if $n$ odd}; \\
	2^{s_2(n)} +2^{s_2(n-2)}, & \text{if $n$ even}.
\end{cases}
$$
\end{theorem}

A similar characterization was stated, without proof or attribution,
in the notes to \texttt{A114212} of the OEIS.

\begin{proof}
From Proposition~\ref{prp_reduction} we get,
for integers $n \geq k \geq 0$, that
\begin{align*}
T(2n,2k) &\equiv T(2n+1,2k) \equiv T(2n+1,2k+1) \equiv \modd{{n \choose k}} {2} ; \\
T(2n,2k+1)&\equiv \modd{ { {n-1} \choose k}} {2} . 
\end{align*}
Then
\begin{align*}
r(2m) &= \sum_{k=0}^{2m} \left( T(2m,k) \bmod 2 \right) \\ 
&=  \sum_{k=0}^m \left( T(2m,2k) \bmod 2 \right)  
\quad + \quad
\sum_{k=0}^{m-1} \left( T(2m, 2k+1) \bmod 2 \right)  \\
&=  \sum_{k=0}^m \left( {m \choose k} \bmod 2 \right) 
\quad + \quad
\sum_{k=0}^{m-1} \left( {{m-1} \choose k} \bmod 2 \right)  \\
&= 2^{s_2 (m)} + 2^{s_2 (m-1)}  \\
&= 2^{s_2 (2m)} + 2^{s_2 (2m-2)}. 
\end{align*}
Similarly,
\begin{align*}
r(2m+1) &= \sum_{k=0}^{2m+1} \left( T(2m+1,k) \bmod 2 \right) \\
&=  \sum_{k=0}^m \left( T(2m+1,2k) \bmod 2 \right) 
\quad + \quad
 \sum_{k=0}^m \left( T(2m+1, 2k+1) \bmod 2 \right)  \\
&=  \sum_{k=0}^m \left( {m \choose k} \bmod 2 \right) 
\quad + \quad
 \sum_{k=0}^m \left( {m \choose k} \bmod 2 \right)   \\
&= 2^{s_2 (m)} + 2^{s_2 (m)} \\
&= 2^{s_2 (2m+1)} .
\end{align*}
\end{proof}
\bibliographystyle{siam}
\bibliography{A114214}
\end{document}